\documentclass[a4paper,12pt]{article}
\usepackage{latexsym,amsmath,amsthm,amssymb}
\usepackage{a4wide}
\usepackage{hyperref}
\usepackage{marginnote}
\usepackage{color}
\usepackage{cite}
\hypersetup{
pdftitle={Orlicz-Lorentz}   
pdfauthor={Van Hoang Nguyen},
colorlinks = true,
linkcolor = magenta,
citecolor = blue,
}

\theoremstyle{plain}
\newtheorem{theorem}{Theorem}[section]

\newtheorem{lemma}[theorem]{Lemma}
\newtheorem{corollary}[theorem]{Corollary}

\theoremstyle{definition}

\theoremstyle{remark}

\renewcommand{\thefootnote}{\arabic{footnote}}

\def\R{\mathbb R}

\def\i0i{\int_0^\infty}

\numberwithin{equation}{section}

\title{Orlicz--Lorentz centroid bodies}
\author{Van Hoang Nguyen\footnote{
Institut de Math\'ematiques de Toulouse, Universit\'e Paul Sabatier, 118 Route de Narbonne, 31062 Toulouse c\'edex 09, France.}
}

\begin{document}
\maketitle

\renewcommand{\thefootnote}{}

\footnote{Email: \href{mailto: Van Hoang Nguyen <van-hoang.nguyen@math.univ-toulouse.fr>}{van-hoang.nguyen@math.univ-toulouse.fr}}

\footnote{2010 \emph{Mathematics Subject Classification\text}: 52A40.}

\footnote{\emph{Key words and phrases\text}:Orlicz--Lorentz centroid body, Orlicz--Lorentz Busemann--Petty centroid inequality, decreasing rearrangement function, Steiner's symmetrization.}

\renewcommand{\thefootnote}{\arabic{footnote}}
\setcounter{footnote}{0}

\begin{abstract}
We extend the definition of the centroid body operator to an Orlicz--Lorentz centroid body operator on the star bodies in $\R^n$, and establish the sharp affine isoperimetric inequality that bounds (from below) the volume of the Orlicz--Lorentz centroid body of any convex body containing the origin in its interior by the volume of this convex body.
\end{abstract}

\section{Introduction}
The concepts of centroid body and projection body are the central notions in convex geometry (or, in Brunn--Minkowski theory). The classical affine isoperimetric inequalities that relate the volume of a convex body with that of its centroid body or its projection body were established in a landmark works of Petty \cite{Petty} and nowaday are known as the Busemann–-Petty centroid inequality and Busemann--Petty projection inequality. (See, e.g., the books of Gardner \cite{Gardner}, Schneider \cite{Schneider}, and Thompson \cite{Thompson} for references.)

The Brunn--Minkowski theory has a natural extension to the $L_p$ Brunn--Minkowski theory and its dual. This new theory was initiated in the early $1960$s when Firey introduced his concept of $L_p$ composition of convex bodies (see, e.g., the book of Schneider \cite{Schneider}). These Firey--Minkowski $L_p$ combinations were shown to lead to an embryonic $L_p$ Brunn--Minkowski theory in the works of Lutwak \cite{Lutwak93,Lutwak96}.  This new theory (and its dual) has witnessed a rapid growth. Its central concepts (and its dual) are the $L_p$ centroid body and $L_p$ projection body (an $L_p$ analogue of the centroid body and projection body) which was introduced by Lutwak, Yang and Zhang \cite{LYZ00}. The $L_p$ analogues of the Busemann--Petty centroid inequality and Busemann--Petty projection inequality were also established in \cite{LYZ00} by using Steiner's symmetrization method, and nowaday are named as the $L_p$ Busemann--Petty centroid inequality and $L_p$ Busemann--Petty projection inequality or shortly $L_p$ affine isoperimetric inequalities (see \cite{CG02,Cordero,Paouris12} for the other proofs of these inequalities based on the shadow system which was introduced by Rogers and Shephard \cite{Rogers,Shephard} or the radom method). It was shown in \cite{Cianchi,LYZ02} that $L_p$ affine isoperimetric inequalities are crucial tools to establish the sharp affine $L_p$ Sobolev inequalities and the affine P\'olya--Szeg\"o principle which are stronger than the usual sharp Sobolev inequalities and the usual P\'oly--Szeg\"o principle in Euclidean space (see \cite{Haberl09,HaSch09,Haberl12} for the strengthened asymmetric counterparts of these results). The $L_p$-centroid bodies recently found some important applications in the field of asymptotic geometric analysis (see, e.g., \cite{Dafnis,Fleury,Gian,Guedon,Klartag,Paouris05,Paouris06a,Paouris06b,Paouris12a} and references therein, especially,in establishing the concentration of mass on convex bodies of Paouris \cite{Paouris06a,Paouris06b} and in thin--shell estimates of Gu\'edon and Milman) and even in the theory of stable distributions \cite{Mol}.

Recently, Lutwak, Yang and Zhang extended the $L_p$ Brunn--Minkowski theory to an Orlicz--Brunn--Minkowski theory by introducing the concepts of Orlicz centroid body (see \cite{LYZcentroid}) and the Orlicz projection body (see \cite{LYZprojection}) for any convex body. They also established the affine isoperimetric inequalities revealing the volume of a convex body with the volume of its Orlicz centroid body and the volume of its Orlicz projection body which are called the Orlicz Busemann--Petty centroid inequality and Orlicz Busemann--Petty projection inequality, repectively (see \cite{Cordero,Paouris12,Li} for the other proofs of these affine isoperimetric inequalities, and see also \cite{Zhu14} for the Orlicz Busemann--Petty centroid inequality on the star bodies). The reverse Orlicz Busemann--Petty centroid inequality was proved in \cite{Chen}. Since the works of Lutwak, Yang and Zhang, the Orlicz--Brunn--Minkowski theory were developed very fast by many authors (see, e.g., \cite{Chen,Gardner14,Gardner15,Haberl12,Huang,Kone,Li16,Xi14,Ye16,Zhu14,Zou14,Zou16} and references therein). For example, in \cite{Gardner14}, Gardner, Hug and Weil developed a general framework for this new theory by introducing the definition of Orlicz addtion. They show that Orlicz addition is intimately related to a natural and fundamental generalization of Minkowski addition called $M$--addition. They also proved some inequalities of Brunn--Minkowski type (such as the Orlicz--Brunn--Minkowski inequality and Orlicz--Minkowski inequality) for both Orlicz addition and $M$-addition. These new inequalities are generalizations of the ones in the $L_p-$Brunn--Minkowski theory, and have a connection with the conjectured log--Brunn--Minkowski inequality of B\"or\"oczky, Lutwak, Yang and Zhang \cite{BLYZlog}. Another proof of the Orlicz--Brunn--Minkowski inequality using Steiner symmetrization method can be found in \cite{Xi14}. In \cite{Haberl12}, Haberl, Lutwak, Yang and Zhang posed the Orlicz--Minkowski problem asking the necessary and sufficient conditions of a given Borel measure on sphere for which this measure is the Orlicz surface area of a convex body. This problem was solved in \cite{Haberl12} when the given measure is even. For the discrete measure, this problem was solved in \cite{Huang}. The dual Orlicz--Brunn--Minkowski theory was recently developed in \cite{Gardner15,Ye16,Zhu14}.

In this paper, we extend the definition of Orlicz centroid body of Lutwak,Yang and Zhang to a more general situation of the Orlicz--Lorentz spaces which are generalization of both Orlicz spaces introduced by Orlicz \cite{O} (see also \cite{Lux}) and Lorentz spaces introduced by Lorentz (see \cite{Lo1,Lo2}). To do this, let us recall some basic elements of these spaces. Let $(\Omega,\Sigma,\mu)$ be a measure space with an $\sigma-$finite, non atom measure $\mu$. For any measurable function $f:\Omega\to \R$, we define the distribution function of $f$ by 
\[
\mu_f(t)=\mu(\{x: |f(x)|>t\}),\qquad \forall\,t>0,
\]
and the decreasing rearrangement of $f$ by
\[
f^{*}(t)=\inf\{\lambda>0: \mu_f(\lambda)\leq t\},
\]
for any $t>0$ (for convention $\inf \emptyset =\infty$).

We denote $I=(0,\mu(\Omega))$. A function $\phi:[0,\infty)\to [0,\infty)$ is called an Orlicz function if $\phi$ is a convex function such that $\phi(t)>0$ if $t>0$, $\phi(0)=0$ and $\lim\limits_{t\to\infty}\phi(t)=\infty$. A function $\omega: I\to (0,\infty)$ is called a weight function if $\omega$ is nonincreasing function which is locally integrable with respect to the Lebesgue measure on $I$ such that $\int_I\omega(t)dt=\infty$ if $I=(0,\infty)$. For an Orlicz function $\phi$ and a weight function $\omega$, we define the Orlicz-Lorentz space $\Lambda_{\phi,\omega}$ on $(\Omega,\Sigma,\mu)$ to be the set of all measurable functions $f$ on $\Omega$ such that
\[
\int_I\phi\left(\frac{f^{*}(t)}{\lambda}\right)\omega(t)dt < \infty,
\]
for some $\lambda > 0$. If the function $f\in \Lambda_{\phi,\omega}$, its Orlicz norm is defined by
\begin{equation}\label{O-LN}
\|f\|_{\Lambda_{\phi,\omega}} =\inf\left\{\lambda > 0: \int_I\phi\left(\frac{f^{*}(t)}{\lambda}\right)\omega(t)dt\leq 1\right\}.
\end{equation}
It is obvious from this definition that if $f$ and $g$ have the same distribution function then $\|f\|_{\Lambda_{\phi,\omega}}=\|g\|_{\Lambda_{\phi,\omega}}$. When $\omega\equiv 1$, the Orlicz--Lorentz space $\Lambda_{\phi,\omega}$ is the Orlicz space. Especially, when $\phi(t)=t^p$ and $\omega\equiv 1$, we obtain the Lebesgue space $L_p(\Omega,\mu)$. When $\phi(t) =t$, we obtain the Lorentz space $\Lambda_\omega$.

Let $K$ be a star body (see section \S2 for precise definition) with respect to the origin in $\mathbb R^n$ with volume $|K|$. We consider the measure space $(\Omega,\Sigma,\mu) = (K, \mathcal B_K, \mu^K)$ here and thereafter $\mathcal B_A$ denotes $\sigma-$algebra of all Lebesgue measurable subset of $A$, and $\mu^A$ denotes the normalized measure on $A$ whose density is $1_A(x)dx/| A|$ for any Lebesgue measurable $A\subset \mathbb R^n$ of positive measure. For any vector $x\in \mathbb R^n$, we define the function $f_{x,K}$ on $K$ by $f_{x,K}(y) = x\cdot y$, with $y\in K$ where $x\cdot y$ denotes the standard inner product of vectors $x$ and $y$ in $\mathbb R^n$. Given an Orlicz function $\phi$ and a weight function $\omega$ on $I =(0,1)$, we define the Orlicz--Lorentz centroid body of $K$ denoted by $\Gamma_{\phi,\omega}K$ to be the convex body in $\mathbb R^n$ whose support function is given by
\[
h(\Gamma_{\phi,\omega}K,x) = \|f_{x,K}\|_{\Lambda_{\phi,\omega}} = \inf\left\{\lambda > 0: \int_0^1\phi\left(\frac{f_{x,K}^{*}(t)}{\lambda}\right)\omega(t)dt\leq 1\right\}.
\]
When $\omega\equiv 1$, our definition of Orlicz--Lorentz centroid body coincides with the definition of Orlicz centroid body given by Lutwak, Yang and Zhang \cite{LYZcentroid} for even convex function $\phi$ in $\mathbb R$. Note that Lutwak, Yang and Zhang  defined the Orlicz centroid body for any function convex function $\phi :\mathbb R\to (0,\infty)$ such that $\phi$ is nonincreasing on $(-\infty,0]$, $\phi$ is nondecreasing on $[0,\infty)$ and one of these monotonicity is strict. Their definition is more general than ours in this case. However, when $\phi(t) =t^p$ and $\omega =1$, we again obtain the defintion of the $L_p$ centroid body given in \cite{LYZ00}.

We will establish the following affine isoperimetric inequality for Orlicz--Lorentz centroid bodies.

\begin{theorem}\label{maintheorem}
If $\phi$ is an Orlicz function, $\omega$ is a weight function on $(0,1)$ and $K$ is a convex body in $\mathbb R^n$ containing the origin in its interior, then the volume ratio 
\[
\frac{|\Gamma_{\phi,\omega}K|}{|K|}
\]
is minimized if and only if $K$ is an origin--centered ellipsoid.
\end{theorem}
This theorem contains as a special case the classical Busemann--Petty centroid inequality for convex bodies \cite{Petty}, as well as the $L_p$ Busemann--Petty centroid inequality for convex bodies (even for star bodies with respect to the origin) that established in \cite{LYZ00}, and the Orlicz Busemann--Petty centroid inequality for convex bodies that established in \cite{LYZcentroid} for even convex function $\phi$. Our proof of Theorem \ref{maintheorem} use the traditional approach to establish the $L_p$ Busemann--Petty centroid inequality \cite{LYZ00} and the Orlicz--Busemann--Petty centroid inequality \cite{LYZcentroid} by using Steiner's symmetrization (see Section \S2 for its definition). However, the appearance of the weight function $\omega$ and working with the decreasing rearrangement function make our proof more complicate. We strongly believe that the shadow system approach (see \cite{CG02,Li}) or radom approach \cite{Cordero,Paouris12} would give the another proof of Theorem \ref{maintheorem}.

The rest of this paper is organized as follows. In Section \S1, we list some basic and well-known facts of conves bodies. Some basic properties of the Orlicz--Lorentz centroid body will be given in Section \S3. Section \S4 is devoted to prove Theorem \ref{maintheorem}.

\section{Background material}
Schneider's book \cite{Schneider} is an excellent reference on theory of convex bodies. Our setting will be Euclidean $n$-space $\mathbb R^n$. We write $e_1,e_2,\ldots,e_n$ for the standard orthonormal basis of $\mathbb R^n$ and when we write $\mathbb R^n = \mathbb R^{n-1} \times \mathbb R$, we always assume that $e_n$ is associated with the last factor. We will attempt to use $x, y$ for vectors in $\mathbb R^n$ and $x', y'$ for vectors in $\mathbb R^{n-1}$. We will also attempt to use $a,b,s,t$ for numbers in $\mathbb R$ and $\mathbb c, \lambda$ for strictly positive reals. If $Q$ is a Borel subset of $\mathbb R^n$ and $Q$ is contained in an $i$-dimensional affine subspace of $\mathbb R^n$ but in no affine subspace of lower dimension, then $|Q|$ will denote the $i$-dimensional Lebesgue measure of $Q$. If $x \in \mathbb R^n$ then by abuse of notation we will write $|x|$ for the norm of $x$. For any $r >0$, we denote by $B_r$ the ball centered at the origin of radius $r$. The unit ball $B_1$ will be written by $B$ for simplicity. Its volume is $\omega_n = |B| =\pi^{n/2}/\Gamma(1 +n/2)$. The unit sphere in $\mathbb R^n$ will be denoted by $S^{n-1}$.

For $A \in GL(n)$ (the set of all invertible $n\times n$ matrices), we write $A^t$ for the transpose of $A$ and $A^{-t}$ for the inverse of the transpose (contragradient) of $A$. Write $|A|$ for the absolute value of the determinant of $A$.

Let $\mathcal C$ denote the set of all Orlicz functions $\phi$ on $[0,\infty)$. It is remarkable from its definition that any Orlicz function $\phi\in \mathcal C$ is strict increasing in $[0,\infty)$, and hence its inversion function $\phi^{-1}$ exists and is continuous. We say that the sequence $\{\phi_i\}_i$ of Orlicz functions is such that $\phi_i\to \phi_0 \in \mathcal C$ if 
\[
|\phi_i-\phi_0|_I = \max_{t\in I} |\phi_i(t) -\phi_0(t)| \to 0,
\]
for any compact interval $I \subset [0,\infty)$. 

A subset $K\subset \mathbb R^n$ is a star-shaped about the origin if for any $x\in K$ then the segmet $\{tx\, :\, t\in [0,1]\}$ is contained in $K$. For a star-shaped about the origin $K$, its radial function $\rho_K: \mathbb R^n \setminus \{0\} \to [0,\infty]$ is defined by
\[
\rho_K(x) = \max\{\lambda >0\, : \lambda x \in K\}.
\]
If $\rho_K$ is strict positive and continuous, then we call $K$ a star body. Let $\mathcal S_0^n$ denote the set of all star bodies with respect to the origin in $\mathbb R^n$. It is obvious that $\rho_{cK} =c\rho_K$ for any $c >0$, where $cK = \{cx\, :\, x\in K\}$. The radial distance between $K,L \in \mathcal S_0^n$ is
\[
|\rho_K -\rho_L|_\infty = \max_{u\in S^{n-1}} |\rho_K(u) -\rho_L(u)|.
\]

A convex body in $\mathbb R^n$ is a compact convex subset of $\mathbb R^n$ with nonempty interior. For any convex body $K$, its support function is defined by
\[
h_K(x)=h(K,x): = \max_{y\in K} x\cdot y.
\]
It is well-known that a convex body is completely determined by its support function. The Hausdorff distance between convex bodies $K$ and $L$ is
\[
|h_K -h_L|_\infty = \max_{u\in S^{n-1}} |h_K(u) -h_L(u)|.
\]
Let $\mathcal K^n$ denote the set of all convex bodies of $\mathbb R^n$ and let $\mathcal K_0^n$ denote the set of all convex bodies containing the origin in its interior of $\mathbb R^n$. Note that on $\mathcal K_0^n$ the radial distance and Hausdorff distance are equivalent.

For $K \in \mathcal S_0^n$, denote
\begin{equation}\label{eq:rKRK}
R_K = \max_{u\in S^{n-1}} \rho_K(u),\qquad r_K = \min_{u\in S^{n-1}} \rho_K(u).
\end{equation}
Since $K \in \mathcal S_0^n$ then $0 < r_K \leq R_K < \infty.$

For a convex body $K$ and a direction $u\in S^{n-1}$, let $K_u$ denote the image of the orthogonal projection of $K$ on $u^\perp$, the subspace of $\mathbb R^n$ orthogonal to $u$. Let $f_u$ and $g_u$ denote the undergraph and overgraph functions of $K$ in the direction $u$, i.e., K is described by
\[
K=\{y'+tu: -f_u(y')\leq t\leq g_u(y');\quad y'\in K_u\}.
\]
Note that $f_u,g_u: K_u\rightarrow \R$ are concave functions. For $y'\in K_u$, we define
\begin{equation}\label{eq:sigmaandm}
\sigma(y')=\frac{f_u(y')+g_u(y')}{2} \quad \text{and}\quad m(y')=\frac{g_u(y')-f_u(y')}{2},
\end{equation}
that is, $\sigma(y')$ is a half of the length of the chord $K\cap \{y'+ \R u\}$, and $y'+m(y')u$ is the midpoint of this chord. With these notations, we have another description of $K$ as follows
\[
K=\{y'+(m(y')+t)u: y'\in K_u, |t|\leq \sigma(y')\}.
\]
The Steiner symmetrization of $K$ in the direction $u$ denoted by $S_uK$ is the convex body defined by
\[
S_uK=\{y'+tu: |t|\leq \sigma(y'), y'\in K_u\}.
\]
It follows from Fubini's theorem that $|S_uK| = |K|$ for any $u\in S^{n-1}$. Moreover, for any convex body $K$, there exists a sequence $\{u_i\}_{i\geq 1} \subset S^{n-1}$ such that the sequence of convex bodies $\{S_{u_i}\cdots S_{u_1}K\}_{i\geq 1}$ converges to an origin-centered ball of volume $|K|$. This is the content of Blaschke's selection theorem.

When considering the convex body $K \subset \mathbb R^{n-1}\times \mathbb R$, for $(x', t) \in \mathbb R^{n-1}\times \mathbb R$ we will usually write $h(K,x',t)$ rather than $h(K; (x', t))$. The following Lemma is, in fact, an immediate consequence of Fubini's theorem.

\begin{lemma}\label{Fub}
Let $K$ be a convex body in $\mathbb R^n$ and $u$ is a direction in $S^{n-1}$. Then the following maps $S:K\rightarrow S_uK$ defined by
\[
S(y'+(m(y')+t)u)=y'+tu,
\]
and $T:K\rightarrow K$ defined by
\[
T(y'+(m(y')+t)u)=y'+(m(y')-t)u
\]
with $y'\in K_u$ and $|t|\leq \sigma(y')$ are volume preserving maps.
\end{lemma}

The following is well known (see \cite{CG02}).

\begin{lemma}\label{graph}
Suppose $K \in \mathcal K^n_0$ and $u \in S^{n-1}$. For any $y'\in  {\rm relint}(K_u)$, the overgraph and undergraph functions of K in direction u are given by
\begin{equation}\label{eq:overgraph}
g_u(y') = \min_{x'\in u^\perp} h(K,x'+u) - x'\cdot y',
\end{equation}
and
\begin{equation}\label{eq:undergraph}
f_u(y') = \min_{x'\in u^\perp} h(K,x'-u) - x'\cdot y'.
\end{equation}
\end{lemma}

The following estimate was proved in \cite{LYZcentroid}.

\begin{lemma}\label{boundonKu}
Suppose $K \in \mathcal K^n_0$ and $u \in S^{n-1}$. If $y'\in (r_K/2)B \cap u^\perp$ and $x'_1, x'_2\in  u^\perp$ are such that
\[
g_u(y') = h(K,x'_1+tu) - x'_1\cdot y'\qquad\text{\rm and}\qquad f_u(y') = h(K,x'_2-tu) - x'_2\cdot y',
\]
then both
\[
|x'_1|,\, |x'_2| \leq \frac{2R_K}{r_K}.
\]
\end{lemma}

Finally, in order to establish the equality case in our Orlicz--Lorentz Busemann--Petty centroid inequality in Theorem \ref{maintheorem}, we need to know which characterizations of $K\in \mathcal K_0^n$ are to being an origin--centered ellipsoid. A classical result says that a convex body $K\in \mathcal K_0^n$ is an origin--centered ellipsoid if and only if for any direction $u\in S^{n-1}$ all of the midpoints of the chords of $K$ parallel to $u$ lie in a subspace of $\mathbb R^n$. In our proof, we need the following characterization of the ellipsoid due to Gruber and Ludwig \cite{Gruber,GruberLudwig}.

\begin{lemma}\label{ellipsoid}
A convex body $K\in \mathcal K^n_0$ is an origin--centered ellipsoid if
and only if there exists an $\epsilon_K > 0$ such that for any direction $u\in S^{n-1}$ all of the chords of K that come within a distance of $\epsilon_K$ of the origin and are parallel to u, have midpoints that lie in a subspace of $\mathbb R^n$.
\end{lemma}

\section{Basic properties of Orlicz--Lorentz centroid bodies}
Recall that if $\phi\in \mathcal C$ is an Orlicz function, $\omega: (0,1)\to (0,\infty)$ is a weight function and $K\in \mathcal S_0^n$, then the Orlicz--Lorentz centroid body of $K$ denoted by $\Gamma_{\phi,\omega}K$ is defined to be a convex body whose support function is
\begin{equation}\label{eq:definitionofOLCB}
h(\Gamma_{\phi,\omega}K,x) = \|f_{x,K}\|_{\Lambda_{\phi,\omega}} = \inf\left\{\lambda >0\, :\, \int_0^1 \phi\left(\frac{f_{x,K}^*(t)}{\lambda}\right) \omega(t) dt \leq 1\right\}.
\end{equation}
Since $\phi\in \mathcal C$, and $\omega$ is strictly positive then $h(\Gamma_{\phi,\omega}K,x) > 0$ for any $x\not=0$.

Note that the function
\[
\Phi: \lambda \in (0,1) \longmapsto  \int_0^1 \phi\left(\frac{f_{x,K}^*(t)}{\lambda}\right) \omega(t) dt
\]
is strictly decreasing, continous on $(0,\infty)$ since $\int_0^1 \omega(t) dt < \infty$. Moreover, it satisfies
\[
\lim_{\lambda \to 0^+} \Phi(\lambda) = \infty,\qquad \lim\limits_{\lambda\to\infty} \Phi(\lambda) =0.
\]
Thus we easily get the following.

\begin{lemma}\label{norm}
Suppose that $K\in \mathcal{S}_0^n$ and $u_0\in S^{n-1}$. Then 
\[
\int_0^1\phi\left(\frac{f_{u_0,K}^{*}(t)}{\lambda_0}\right)\omega(t)dt=1
\]
if and only if 
\[
h(\Gamma_{\phi,\omega}K,u_0)=\lambda_0.
\]
\end{lemma}

Since $\|\cdot\|_{\Lambda_{\phi,\omega}}$ is a norm on $\Lambda_{\phi,\omega}$, thus we have:

\begin{lemma}
If $K\in\mathcal{S}_0^n$ then $h(\Gamma_{\phi,\omega}K,.)$ is the support function of an origin--centered convex body in $\mathcal{K}_o^n$.
\end{lemma}

Next lemma gives us the upper and lower bounds for the support function $h(\Gamma_{\phi,\omega}K,.)$. 
\begin{lemma}\label{boundsf}
If $K\in\mathcal{S}_o^n$ then
\[
\frac{1}{r_Kf_{u,B}^{*}(1/2)\phi^{-1}\left(\frac{1}{\int_0^{c(n,K)}\omega(t)dt}\right)}\leq h(\Gamma_{\phi,\omega}K,u)\leq \frac{R_K}{\phi^{-1}\left(\frac{1}{\int_0^1\omega(t)dt}\right)}
\]
for any $u\in S^{n-1}$, where $r_K, R_K$ is defined by \eqref{eq:rKRK}, $\phi^{-1}$ denotes the inverse function of $\phi$, $f_{u,B}^{*}$ is the decreasing rearangement function of the function $x\mapsto  u \cdot x$ on $B$ and is defined on the measure space $(B,\mathcal B_B, \mu^B)$ and
\[
c(n,K)=\frac{r_K^n\omega_n}{2|K|}.
\]
\end{lemma}
\begin{proof}
We follow the argument in the proof of Lemma $2.3$ in \cite{LYZcentroid}. Given $u\in S^{n-1}$ and suppose that $h(\Gamma_{\phi,\omega}K,u)=\lambda_0$, by Lemma~\ref{norm} we have
\begin{equation}\label{O-LN1}
\int_0^1\phi\left(\frac{f_u^{*}(t)}{\lambda_0}\right)\omega(t)dt=1.
\end{equation}
We first prove the upper bound. By the definition of $R_K$, we have $K\subset R_KB$ which implies $|f_{u,K}(x)|\leq R_K$ for any $x\in K$, hence $f_{u,K}^{*}(t)\leq R_K$ for any $t\in (0,1)$. Thus, by \eqref{O-LN1} and the strict increasing monotonicity of $\phi$, we obtain
\[
\phi\left(\frac{R_K}{\lambda_0}\right)\int_0^1\omega(t)dt \geq 1,
\]
or equivalently,
\[
\lambda_0\leq \frac{R_K}{\phi^{-1}\left(\frac{1}{\int_0^1\omega(t)dt}\right)}.
\]

We next prove the lower bound. By the definition of $r_K$, we have $r_KB\subset K$ then we have
\[
\{x\in K: |x\cdot u|>t\}\supset r_K\left\{x\in B: |x\cdot u|>\frac{t}{r_K}\right\},\qquad\forall\, t>0,
\]
which then implies 
\[
\mu^K_{f_{u,K}}(t)\geq \frac{r_K^n\omega_n}{|K|}\mu^B_{f_{u,B}}\left(\frac{t}{r_K}\right),\qquad\forall\, t>0.
\]
Thus, by the definition of the decreasing rearrangement function, we readily obtain
\[
f_{u,K}^{*}(t)\geq 
\begin{cases}
r_Kf_{u,B}^{*}(\frac{|K|}{r_K^n\omega_n}t)&\mbox{if $t< \frac{r_K^n \omega_n}{|K|}$,}\\
0&\mbox{if $\frac{r_K^n \omega_n}{|K|}\leq t <1$.}
\end{cases}
\]
Denote $c(n,K) = r_K^n \omega_n /(2 |K|)$ we then have
\begin{equation}\label{O-LN2}
f_{u,K}^{*}(t)\geq r_Kf_{u,B}^{*}\left(\frac{1}{2}\right),\qquad \forall\,t\in(0, c(n,K)].
\end{equation}
It follows from \eqref{O-LN1},~\eqref{O-LN2} and the strictly increasing monotonicity of $\phi$ that
\[
\phi\left(\frac{r_Kf_{u,B}^{*}(1/2)}{\lambda_0}\right)\int_0^{c(n,K)}\omega(t)dt\leq 1,
\]
or equivalently,
\[
\lambda_0\geq \frac{1}{r_Kf_{u,B}^{*}(1/2)\phi^{-1}\left(\frac{1}{\int_0^{c(n,K)}\omega(t)dt}\right)}.
\]
\end{proof}
Since $f_{u,B}^*$ does not depend on $u\in S^{n-1}$, hence Lemma \ref{boundsf} gives a lower bound of $h(\Gamma_{\phi,\omega}K,u)$ which is independent of $u\in S^{n-1}$. In the next lemma, we show that the Orlicz-Lorentz centroid operator $\Gamma_{\phi,\omega}$ commutes with any $A\in GL(n)$.
\begin{lemma}\label{commute}
If $K\in \mathcal{S}_o^n$ and $A\in GL(n)$ then
\[
\Gamma_{\phi,\omega}(AK) = A\Gamma_{\phi,\omega}K.
\]
\end{lemma}
\begin{proof}
Let $u\in \R^n$, it is evident that $f_{u,AK}(x)=f_{A^tu,K}(A^{-1}x)$ for any $x\in \mathbb R^n$. Hence 
\[
\{x\in AK: |f_{u,AK}(x)|>s\}=A\{x\in K: |f_{u,K}(x)|>s\},
\]
which then implies $\mu^{AK}_{f_{u,AK}}(s)=\mu^K_{f_{A^tu,K}}(s)$ for any $s > 0$. Consequently, we get
\[
f_{u,AK}^{*}(t)=f_{A^tu,K}^{*}(t),\qquad\forall\, t\in (0,1).
\]
The definition of the Orlicz--Lorentz centroid body \eqref{eq:definitionofOLCB} then yields
\[
h(\Gamma_{\phi,\omega}(AK),u)=h(\Gamma_{\phi,\omega}K,A^tu)=h(A\Gamma_{\phi,\omega}K,u),
\]
for any $u\in \mathbb R^n$. This finishes our proof.
\end{proof}

We next prove that the Orlicz-Lorentz centroid operator $\Gamma_{\phi,\omega}: \mathcal{S}_0^n\longrightarrow \mathcal{K}_0^n$ is continuous.

\begin{lemma}\label{continue1}
Let $\phi\in \mathcal C$ be an Orlicz function and $\omega$ is a weight function on $(0,1)$. If $K_i, K\in \mathcal{S}_0^n$, $i\geq 1$ and $K_i\to K$ in $\mathcal{S}_0^n$ then $\Gamma_{\phi,\omega}K_i\to \Gamma_{\phi,\omega}K$ in $\mathcal K_0^n$.
\end{lemma}
\begin{proof}
Since $\Gamma_{\phi,\omega}K_i, \Gamma_{\phi,\omega}K \in \mathcal K_0^n$, $i\geq 1$, it is enough to show that
\[
\lim_{i\to\infty} h(\Gamma_{\phi,\omega}K_i,u) = h(\Gamma_{\phi,\omega}K,u),
\]
for any $u\in S^{n-1}$. Fix a vector $u\in S^{n-1}$, denote
\[
\lambda_i=h(\Gamma_{\phi,\omega}K_i,u).
\]
Lemma~\ref{boundsf} says that
\[
\frac{1}{r_{K_i}f_{u,B}^{*}(1/2)\phi^{-1}\left(\frac{1}{\int_0^{c(n,K_i)}\omega(t)dt}\right)}\leq \lambda_i\leq \frac{R_{K_i}}{\phi^{-1}\left(\frac{1}{\int_0^1\omega(t)dt}\right)},\qquad\forall\, i\geq 1.
\]
Since $K_i\to K\in \mathcal{S}_o^n$ then we have 
\[
|K_i|\to |K|, \qquad r_{K_i}\to r_K > 0,\qquad {\rm and}\qquad R_{K_i}\to R_K < \infty.
\]
Thus, we get $c(n,K_i)\to c(n,K) >0 $. Hence there exists positive constants $a,b$ such that 
\begin{equation}\label{eq:boundlambdai}
a\leq \lambda_i\leq b,\qquad \forall\, i\geq 1. 
\end{equation}

We next show that 
\begin{equation}\label{eq:convergence}
f_{u,K_i}^{*}(t)\to f_{u,K}^{*}(t)\qquad \text{\rm for a.e}\,\, t\in (0,1).
\end{equation}
For $s>0$, denote
\[
A_i(s)=\{x\in K_i: |u.x|>s\}\qquad\text{and}\qquad A(s)=\{x\in K: |u.x|>s\}.
\]
Let $A\Delta B$ denote the symmetric difference of two measurable subsets $A,B\subset \mathbb R^n$, i.e., $A\Delta B = (A\setminus B) \cup (B\setminus A)$. We claim that   
\begin{equation}\label{eq:claim}
A_i(s)\Delta A(s)\subset K_i\Delta K,\qquad\forall\, i\geq 1,
\end{equation}
Indeed, if $x \in A_i(s) \setminus A(s)$, we must have $x\in K_i$ and $|u\cdot x| > s$. This implies that $x\not \in K$ since if $x\in K$ then $x\in A(s)$ which is a contradiction. Hence $x\in K_i\setminus K$. Similarly, if $x\in A(s)\setminus A_i(s)$ then $x\in K\setminus K_i$. Our claim \eqref{eq:claim} is proved.

By our assumption $K_i\to K$ in $\mathcal S_0^n$, we have $|K_i \Delta K|\to 0$. This fact and our claim \eqref{eq:claim} yield 
\begin{equation}\label{eq:limit0}
\lim_{i\to\infty} |A_i(s)\Delta A(s)| = 0.
\end{equation}

Let $t\in (0,1)$ be an arbitrary point. For any $s < f_{u,K}^*(t)$, we must have $\mu^K_{f_{u,K}}(s) > t$. It is obvious that $A(s) \subset A_i(s) \cup (A(s)\setminus A_i(s))$, then
\[
|A(s)|\leq |A_i(s)| + |A(s) \setminus A_i(s)|.
\]
Combining this inequality and \eqref{eq:limit0} proves that
\[
\liminf_{i\to\infty}\mu^{K_i}_{f_{u,K_i}}(s) = \liminf_{i\to\infty} \frac{|A_i(s)|}{|K_i|} \geq \frac{|A(s)|}{|K|} = \mu^K_{f_{u,K}}(s) > t.
\]
Therefore, there exists $i_0$ such that 
\[
\mu_{f_{u,K_i}}^{K_i}(s) > t,\qquad\forall\, i\geq i_0.
\]
This implies 
\[
f_{u,K_i}^*(t) \geq s,\qquad\forall\, i\geq i_0,
\]
and hence
\[
\liminf_{i\to \infty} f_{u,K_i}^*(t) \geq s.
\]
Since $s < f_{u,K}^*(t)$ is arbitrary, then 
\begin{equation}\label{eq:liminf}
\liminf_{i\to \infty} f_{u,K_i}^*(t) \geq f_{u,K}^*(t).
\end{equation}

Let $t\in (0,1)$ be a left continuous point of $f_{u,K}^*$. For any $s > f_{u,K}^*(t)$, there exists $t'\in (0,t)$ such that $f_{u,K}^*(t') < s$ by the left continuity of $f_{u,K}^*$ at $t$. Consequently, by the definition of $f_{u,K}^*$, we have
\[
\mu_{f_{u,K}}(s) \leq t' < t.
\]
Note that $A_i(s)\subset A(s)\cup (A_i(s)\setminus A(s))$, hence
\[
|A_i(s)|\leq |A(s)| + |A_i(s)\setminus A(s)|.
\]
Combining this inequality and \eqref{eq:limit0} proves that
\[
\limsup_{i\to\infty} \mu^{K_i}_{f_{u,K_i}}(s) = \limsup_{i\to\infty} \frac{|A_i(s)|}{|K_i|} \leq \frac{|A(s)|}{|K|} = \mu_{f_{u,K}}(s) < t,
\] 
here we use \eqref{eq:limit0}. Thus, there exists $i_1$ such that
\[
\mu^{K_i}_{f_{u,K_i}}(s) < t, \qquad\forall \, i\geq i_1.
\]
By definition of $f_{u,K_i}^*$, we obtain
\[
f_{u,K_i}^*(t) \leq s,\qquad\forall\, i\geq i_1,
\]
which implies
\[
\limsup_{i\to\infty}f_{u,K_i}^*(t) \leq s.
\]
Since $s > f_{u,K}^*(t)$ is arbitrary, then 
\begin{equation}\label{eq:limsup}
\limsup_{i\to\infty}f_{u,K_i}^*(t) \leq f_{u,K}^*(t).
\end{equation}

\eqref{eq:liminf} and \eqref{eq:limsup} show that $f_{u,K_i}^{*}(t)\to f_{u,K}^{*}(t)$ for any left continuous point $t$ of $f_{u,K}^{*}$. This proves \eqref{eq:convergence} since $f_{u,K}^{*}$  is left continuous a.e in $(0,1)$ because of the nonincreasing monotonicity.

Let $\{\lambda_{i_k}\}_k$ be an arbitrary subsequence of $\{\lambda_i\}_i$. Since $\{\lambda_{i_k}\}_k$ is bounded (by $a$ and $b$, see \eqref{eq:boundlambdai}), it possesses a subsequence (still denoted by $\{\lambda_{i_k}\}_k$) converging to $\lambda_0$. Obviously, we have $a\leq \lambda_0\leq b$. We have from definition of $h(\Gamma_{\phi,\omega}K_{i_k}, u)$ that
\[
\int_0^1 \phi\left(\frac{f_{u,K_{i_k}}^*(t)}{\lambda_{i_k}}\right) \omega(t) dt =1,
\]
for any $k\geq 1$. Since $R_0 = \sup\{R_i\, :\, i\geq 1\} < \infty$, hence
\[
\phi\left(\frac{f_{u,K_{i_k}}^*(t)}{\lambda_{i_k}}\right) \omega(t) \leq \phi\left(\frac{R_0}{a}\right)\omega(t) \in L_1((0,1)).
\]
Letting $k\to\infty$ and using the the dominated convergent theorem and \eqref{eq:convergence}, we get 
\[
\int_0^1\phi\left(\frac{f_{u,K}^{*}(t)}{\lambda_0}\right)\omega(t)dt=1,
\]
or $\lambda_0 = h(\Gamma_{\phi,\omega}K,u)$ by Lemma \ref{norm}. Since $\{\lambda_{i_k}\}_k$ is an arbitrary subsequence of $\{\lambda_i\}_i$, then we have
\[
\lim_{i\to\infty} h(\Gamma_{\phi,\omega}K_i,u) = h(\Gamma_{\phi,\omega}K, u).
\]
This Lemma is completely proved.
\end{proof}
We next show that the Orlicz-Lorentz centroid operator is continous in $\phi$. Recall that $\phi_i\to \phi\in \mathcal{C}$ if $\phi_i$ uniformly converges to $\phi$ on any compact interval of $[0,\infty)$.
\begin{lemma}\label{continue2}
If $\phi_i\to \phi\in \mathcal{C}$, then $\Gamma_{\phi_i,\omega}K\to \Gamma_{\phi,\omega}K$ for any $K\in \mathcal{S}_0^n$ and the weight function $\omega$ on $(0,1)$ .
\end{lemma}
\begin{proof}
Suppose that $K\in \mathcal{S}_0^n$, it is enough to prove that 
\begin{equation}\label{eq:enough}
h(\Gamma_{\phi_i,\omega}K,u)\to h(\Gamma_{\phi,\omega}K,u)
\end{equation}
for any $u\in S^{n-1}$. Denote
\[
\lambda_i=h(\Gamma_{\phi_i,\omega}K,u),\qquad \forall\, i\geq 1.
\]
It implies from Lemma~\ref{boundsf} that
\begin{equation}\label{eq:abcd}
\frac{1}{r_{K}f_{u,B}^{*}(1/2)\phi_i^{-1}\left(\frac{1}{\int_0^{c(n,K)}\omega(t)dt}\right)}\leq \lambda_i\leq \frac{R_{K}}{\phi_i^{-1}\left(\frac{1}{\int_0^1\omega(t)dt}\right)},\qquad\forall\, i\geq 1.
\end{equation}

We first prove that 
\begin{equation}\label{lim}
\phi_i^{-1}(a)\to \phi^{-1}(a)
\end{equation}
for all $a>0$. Indeed, we can choose $M,m >0$ such that
\[
0< 2\phi(m) < a < \frac12 \phi(M)< \infty.
\]
Since $\phi_i\to \phi$ in $\mathcal C$ then there exists $i_0$ such that
\[
\phi_i(m) < a < \phi_i(M),\qquad\forall\, i\geq i_0,
\]
or equivalently
\[
m < \phi_i^{-1}(a) < M,\qquad\forall\, i\geq i_0.
\]
Hence $\{\phi_i^{-1}(a)\}_i$ is bounded. Suppose that $\{\phi_{i_k}^{-1}(a)\}_k$ is a subsequence of $\{\phi_i^{-1}(a)\}$ and converges to $b$. Obviously, we have $m\leq b\leq M$. Since $\phi_i\to\phi$ in $\mathcal C$, then
\[
a = \lim_{k\to\infty} \phi_{i_k}(\phi_{i_k}^{-1}(a)) = \phi(b),
\]
or $b=\phi^{-1}(a)$. We have shown that $\{\phi_i^{-1}(a)\}_i$ has at most one accumulation point which is $\phi^{-1}(a)$ if it exists. Since $\{\phi_i^{-1}(a)\}_i$ is bounded, then \eqref{lim} holds.

Combining \eqref{eq:abcd} and \eqref{lim} implies the existence of  $a,b>0$ such that $a < \lambda_i < b$ for any $i\geq 1$. Suppose that $\{\lambda_{i_k}\}_k$ is a subsequence of $\{\lambda_i\}_i$ and converges to $\lambda_0$. From the definition of $h(\Gamma_{\phi_i,\omega}K,u)$ we have
\[
\int_0^1 \phi_{i_k}\left(\frac{f_{u,K}^*(t)}{\lambda_{i_k}}\right) \omega(t) dt =1,
\]
for any $k\geq 1$. Moreover, it holds
\[
\phi_{i_k}\left(\frac{f_{u,K}^*(t)}{\lambda_{i_k}}\right) \omega(t) \leq \phi_{i_k}\left(\frac{R_K}{a}\right) \omega(t) \leq \left(\sup_{i\geq 1} \phi_i\left(\frac{R_K}{a}\right)\right) \omega(t) \in L_1((0,1)).
\]
Letting $k\to\infty$ and using the dominated convergence theorem and the assumption $\phi_i\to\phi$ in $\mathcal C$, we get
\[
\int_0^1 \phi\left(\frac{f_{u,K}^*(t)}{\lambda_0}\right) \omega(t) dt,
\]
or $\lambda_0 = h(\Gamma_{\phi,\omega}K,u)$ by Lemma \ref{norm}. We thus have shown that the sequence $\{\lambda_i\}_i$ has at most one accumulation point which is $h(\Gamma_{\phi,\omega}K,u)$ if it exists. Since $\{\lambda_i\}_i$ is bounded, then \eqref{eq:enough} holds as desired. This finishes our proof.
\end{proof}

\section{Proof of Theorem \ref{maintheorem}}
The next lemma plays crucial role in our proof of Theorem \ref{maintheorem}.

\begin{lemma}\label{main}
Let $\phi\in \mathcal C$ be an Orlicz function, $\omega$ is a weight function on $(0,1)$, and $K\in\mathcal{K}_0^n$. If $u\in S^{n-1}$ and $x'_1,x'_2\in u^{\bot}$, then
\begin{equation}\label{Steiner}
h\left(\Gamma_{\phi,\omega}(S_uK),\frac{1}{2}x'_1+\frac{1}{2}x'_2+u\right)\leq \frac{1}{2}h(\Gamma_{\phi,\omega}K,x'_1+u)+\frac{1}{2}h(\Gamma_{\phi,\omega}K,x'_2-u).
\end{equation}
Equality in \eqref{Steiner} implies that all of the chords of $K$ parallel to $u$, whose distance from the origin is less than $r_K/(2\max\{1,|x'_1|,|x'_2|\})$ have the midpoints that lie in the subspace
\[
\left\{y'+\frac{1}{2}(x'_2-x'_1).y'u\,:\, y'\in u^{\bot}\right\}
\]
of $\mathbb R^n$.
\end{lemma}
\begin{proof}
By the Lemma~\ref{commute} we can assume, without of loss generality, that $|K|=|S_uK|=1$. Let $K_u$ denote the image of the orthogonal projection of $K$ onto the subspace $u^{\bot}$. If $y'\in K_u$, define $\sigma(y')$ and $m(y')$ as in \eqref{eq:sigmaandm}. Denote
\[
x_1=x'_1+u,\quad x_2=x'_2-u,\quad x=\frac{1}{2}x'_1+\frac{1}{2}x'_2+u,\quad\text{and}\quad \lambda_i=h(\Gamma_{\phi,\omega}K,x_i),\quad i=1,2.
\]
From the Lemma~\ref{norm} we have
\[
\int_0^1\phi\left(\frac{f_{x_i,K}^{*}(t)}{\lambda_i}\right)\omega(t)dt=1,\quad i=1,2.
\]
If $y=y'+(m(y')+t)u\in K$ then 
\[
f_{x_1,K}(y)=x'_1.y'+m(y')+t,\quad f_{x_2,K}(Ty)=x'_2.y'-m(y')+t,\quad \text{and}\quad Sy=y'+tu,
\]
where $S,T$ are maps given in the Lemma~\ref{Fub}. Hence we have
\[
f_{x,S_uK}(Sy)=\frac12(x_1'\cdot y' + x_2'\cdot y') + t = \frac{1}{2}f_{x_1,K}(y)+\frac{1}{2}f_{x_2,K}(Ty)=:f(y).
\]
The volume preserving property of $S$ (by Lemma~\ref{Fub}) yields $f_{x,S_uK}^{*}=f^{*}$. Similarly, the volume preserving property of $T$ implies that $f_{x_2,K}\circ T$ and $f_{x_2,K}$ have the same decreasing rearrangement function. 

Denote $\lambda=(\lambda_1+\lambda_2)/2$. Since $\phi\in \mathcal C$, then  $\phi(g^{*})=(\phi(|g|))^{*}$ holds for any measurable function $g$. This identity implies 
\begin{equation}\label{star}
\phi\left(\frac{f_{x,S_uK}^{*}}{\lambda}\right)=\phi\left(\frac{(f_{x_1,K}+f_{x_2,K}\circ T)^{*}}{\lambda_1+\lambda_2}\right)=\left(\phi\left(\frac{|f_{x_1,K}+f_{x_2,K}\circ T|}{\lambda_1+\lambda_2}\right)\right)^{*}. 
\end{equation}
It is obvious that
\begin{equation}\label{star2}
\frac{|f_{x_1,K}+f_{x_2,K}\circ T|}{\lambda_1+\lambda_2}\leq \frac{\lambda_1}{\lambda_1+\lambda_2}\frac{|f_{x_1,K}|}{\lambda_1}+\frac{\lambda_2}{\lambda_1+\lambda_2}\frac{|f_{x_2,K}\circ T|}{\lambda_2}.
\end{equation}
The increasing monotonicity and convexity of $\phi$ together \eqref{star2} imply
\begin{equation}\label{star3}
\phi\left(\frac{|f_{x_1,K}+f_{x_2,K}\circ T|}{\lambda_1+\lambda_2}\right)\leq \frac{\lambda_1}{\lambda_1+\lambda_2}\phi\left(\frac{|f_{x_1,K}|}{\lambda_1}\right)+\frac{\lambda_2}{\lambda_1+\lambda_2}\phi\left(\frac{|f_{x_2,K}\circ T|}{\lambda_2}\right).
\end{equation}
The decreasing rearrangement preserves the order on the positive functions. This fact together \eqref{star} and \eqref{star3} prove that
\begin{equation}\label{star4}
\phi\left(\frac{f_{x,S_uK}^{*}}{\lambda}\right)\leq \left(\frac{\lambda_1}{\lambda_1+\lambda_2}\phi\left(\frac{|f_{x_1,K}|}{\lambda_1}\right)+\frac{\lambda_2}{\lambda_1+\lambda_2}\phi\left(\frac{|f_{x_2,K}\circ T|}{\lambda_2}\right)\right)^{*}.
\end{equation}
Multiplying both sides of \eqref{star4} by $\omega$, then integrating the obtained inequality on $(0,1)$ and using the known fact 
\begin{equation}\label{star5}
\int_0^1(g_1+g_2)^{*}(t)\omega(t)dt\leq \int_0^1g_1^{*}(t)\omega(t)dt+\int_0^1g_2^{*}(t)\omega(t)dt,
\end{equation}
and again the equality $\phi(g^{*})=(\phi(|g|))^{*}$, we obtain
\begin{align}\label{eq:1111}
\int_0^1&\phi\left(\frac{f_{x,S_uK}^{*}(t)}{\lambda}\right)\omega(t)dt\notag\\
&\leq\int_0^1\left(\frac{\lambda_1}{\lambda_1+\lambda_2}\phi\left(\frac{|f_{x_1,K}|}{\lambda_1}\right)+\frac{\lambda_2}{\lambda_1+\lambda_2}\phi\left(\frac{|f_{x_2,K}\circ T|}{\lambda_2}\right)\right)^*(t)\omega(t)dt\notag\\
&\leq\frac{\lambda_1}{\lambda_1+\lambda_2}\int_0^1\phi\left(\frac{f_{x_1,K}^{*}(t)}{\lambda_1}\right)\omega(t)dt+\frac{\lambda_2}{\lambda_1+\lambda_2}\int_0^1\phi\left(\frac{f_{x_2,K}^{*}(t)}{\lambda_2}\right)\omega(t)dt\notag\\
&=1,
\end{align}
here we used the property that $f_{x_2,K}$ and $f_{x_2,K}\circ T$ have the same decreasing rearrangement function. The latter inequality \eqref{eq:1111} and the definition of $h(\Gamma_{\phi,\omega}S_uK, \cdot)$ prove
\[
h(\Gamma_{\phi,\omega}(S_uK),x)\leq \lambda=\frac{1}{2}h(\Gamma_{\phi,\omega}K,x_1)+\frac{1}{2}h(\Gamma_{\phi,\omega}K,x_2),
\]
as our desired inequality \eqref{Steiner}.

Suppose that equality holds in \eqref{Steiner}. Thus we must have equalities in \eqref{star2}, \eqref{star3} for a.e  $y\in K$ and equality in \eqref{star5} for $f_{x_1,K}$ and $f_{x_2,K}\circ T$. The continuity of $f_{x_1,K}$ and $f_{x_2,K}\circ T$ on $K$ implies that \eqref{star2} holds on whole $K$. Hence, for any fixed $y'\in K_u$, the signs of $x'_1\cdot y'+m(y')+t$ and $x'_2\cdot y'-m(y')+t$ coincide for all $|t|\leq \sigma(y')$. 

For any $y'\in K_u$ and $|y'|\leq \frac{r_K}{2}$, we have
\[
y' \pm \frac{\sqrt{3}}2 r_K u\in r_K B \subset K,
\]
hence $g_u(y') \geq \sqrt{3}r_K/2> r_K/2$ and $-f_u(y') \leq -\sqrt{3}r_K/2 <-r_K/2$. Thus we have prove the following inclusion
\begin{equation}\label{include1}
\left(-\frac{r_K}{2},\frac{r_K}{2}\right)\subset (m(y')-\sigma(y'),m(y')+\sigma(y'))
\end{equation}
and hence 
\begin{equation}\label{include2}
\left(-\frac{r_K}{2},\frac{r_K}{2}\right)\subset (-m(y')-\sigma(y'),-m(y')+\sigma(y'))
\end{equation}

Suppose that $y'\in K_u$ and 
\[
|y'|\leq \frac{r_K}{2\max\{1,|x'_1|,|x'_2|\}}.
\]
The inclusions \eqref{include1} and \eqref{include2} imply 
\[
x_1'.y'+m(y')\in (-\sigma(y'),\sigma(y'))\quad \text{and}\quad x'_2.y'-m(y')\in (-\sigma(y'),\sigma(y')).
\]
Hence the affine functions
\[
t\mapsto x_1'.y'+m(y')+t \quad\text{\rm and}\quad t\mapsto x'_2.y'-m(y')+t
\]
both have their root in $(-\sigma(y'),\sigma(y'))$. However, we know that they have the same sign on this interval, then they must have a root at the same $t(y')\in (-\sigma(y'),\sigma(y'))$. This assertion yields 
\[
(x'_2-x'_1).y'=2m(y').
\]
Therefore, we have proved that for any $y'\in K_u$ with $|y'|\leq r_K/(2\max\{1,|x'_1|,|x'_2|\})$, the midpoints 
\[
\{y'+m(y')u: y'\in K_u\}
\]
of the chords of $K$ parallel to $u$ lie in the subspace
\[
\left\{y'+\frac{1}{2}(x'_2-x'_1).y'u: y'\in u^{\bot}\right\}
\]
of $\R^n$ as our desired.
\end{proof}

From the inequality \eqref{Steiner}, we deduce the following inequality
\[
h\left(\Gamma_{\phi,\omega}(S_uK),\frac{1}{2}x'_1+\frac{1}{2}x'_2-u\right)\leq \frac{1}{2}h(\Gamma_{\phi,\omega}K,x'_1+u)+\frac{1}{2}h(\Gamma_{\phi,\omega}K,x'_2-u).
\]

\begin{lemma}\label{main1}
Let $\phi\in \mathcal{C}$ be an Orlicz function, $\omega$ is a weight function on $(0,1)$ and $K\in \mathcal{K}_0^n$. If $u\in S^{n-1}$ then
\begin{equation}\label{include3}
\Gamma_{\phi,\omega}(S_uK)\subset S_u(\Gamma_{\phi,\omega}K).
\end{equation}
If the inclusion is an identity then all chords of $K$ parallel to $u$, whose distance from the origin is less than $r_Kr_{\Gamma_{\phi,\omega}K}/(4R_{\Gamma_{\phi,\omega}K})$ have the midpoints that lie in a subspace of $\R^n$.
\end{lemma}
\begin{proof}
For any compact subset $L \subset \R^n$ and any unit vector $u\in S^{n-1}$, we denote by $L_u$ the orthogonal image of $L$ on $u^{\bot}$. For any $y'\in L_u$, we define
\[
g_u(L, y')=\sup\{t: y'+tu\in L\}
\]
and
\[
f_u(L,y')=-\inf\{t: y'+tu\in L\}=\sup\{-t: y'+tu\in L\}.
\]

Given $y'\in \text{\rm relint}(\Gamma_{\phi,\omega}K)_u$, by Lemma \ref{graph},  there exist $x'_1,x'_2\in u^{\bot}$ such that
\begin{equation}\label{supfunc}
g_u(\Gamma_{\phi,\omega}K,y')=h(\Gamma_{\phi,\omega}K,x'_1+u)-x'_1.y',
\end{equation}
and
\begin{equation}\label{inffunc}
f_u(\Gamma_{\phi,\omega}K,y')=h(\Gamma_{\phi,\omega}K,x'_2-u)-x'_2.y'.
\end{equation}
Combining \eqref{supfunc} and \eqref{inffunc} together Lemma~\ref{main} imply
\begin{align*}
g_u(S_u(\Gamma_{\phi,\omega}K),y')&=\frac{1}{2}(g_u(\Gamma_{\phi,\omega}K,y')+f_u(\Gamma_{\phi,\omega}K,y'))\\
&=\frac{1}{2}(h(\Gamma_{\phi,\omega}K,x'_1+u)+h(\Gamma_{\phi,\omega}K,x'_2-u))-\frac{1}{2}(x'_1+x'_2).y'\\
&\geq h\left(\Gamma_{\phi,\omega}(S_uK),\frac{1}{2}(x'_1+x'_2)+u\right)-\frac{1}{2}(x'_1+x'_2).y'\\
&\geq \min_{x'\in u^{\bot}}\{h(\Gamma_{\phi,\omega}(S_uK),x'+u)-x'.y'\}\\
&=g_u(\Gamma_{\phi,\omega}(S_uK),y'),
\end{align*}
and
\begin{align*}
f_u(S_u(\Gamma_{\phi,\omega}K),y')&=\frac{1}{2}(g_u(\Gamma_{\phi,\omega}K,y')+f_u(\Gamma_{\phi,\omega}K,y'))\\
&=\frac{1}{2}(h(\Gamma_{\phi,\omega}K,x'_1+u)+h(\Gamma_{\phi,\omega}K,x'_2-u))-\frac{1}{2}(x'_1+x'_2).y'\\
&\geq h\left(\Gamma_{\phi,\omega}(S_uK),\frac{1}{2}(x'_1+x'_2)-u\right)-\frac{1}{2}(x'_1+x'_2).y'\\
&\geq \min_{x'\in u^{\bot}}\{h(\Gamma_{\phi,\omega}(S_uK),x'-u)-x'.y'\}\\
&=f_u(\Gamma_{\phi,\omega}(S_uK),y').
\end{align*}
These two inequalities prove the inclusion \eqref{include3}.

Now suppose that the inclusion \eqref{include3} is an identity, then for any $y'\in (\Gamma_{\phi,\omega}K)_u$ we have
\begin{equation}\label{equa}
g_u(S_u(\Gamma_{\phi,\omega}K),y')=g_u(\Gamma_{\phi,\omega}(S_uK),y')\,\text{ and }\,  g_u(S_u(\Gamma_{\phi,\omega}K),y')=f_u(\Gamma_{\phi,\omega}(S_uK),y').
\end{equation}
For each $y'\in (\Gamma_{\phi,\omega}K)_u$ and $|y'|\leq r_{\Gamma_{\phi,\omega}K}/2$, by Lemma \ref{graph}, there exist $x'_1,x'_2\in u^{\bot}$ such that
\[
g_u(\Gamma_{\phi,\omega}K,y')=h(\Gamma_{\phi,\omega}K,x'_1+u)-x'_1.y',
\]
and
\[
f_u(\Gamma_{\phi,\omega}K,y')=h(\Gamma_{\phi,\omega}K,x'_2-u)-x'_2.y'.
\] 
Lemma \ref{boundonKu} implies 
\[
|x'_1|\leq \frac{2R_{\Gamma_{\phi,\omega}K}}{r_{\Gamma_{\phi,\omega}K}}\quad \text{and}\quad |x'_2|\leq \frac{2R_{\Gamma_{\phi,\omega}K}}{r_{\Gamma_{\phi,\omega}K}}.
\]
Equalities in \eqref{equa} deduce that
\[
h\left(\Gamma_{\phi,\omega}(S_uK),\frac{1}{2}(x'_1+x'_2)+u\right)=\frac{1}{2}h(\Gamma_{\phi,\omega}K,x'_1+u)+\frac{1}{2}h(\Gamma_{\phi,\omega}K,x'_2-u).
\]
By Lemma~\ref{main}, all the chords of $K$ parallel to $u$ whose distance from the origin is less than $r_K/(2\max\{1,|x'_1|,|x'_2|\})$ have midpoints that lie in a subspace of $\R^n$. However, the following estimate
\[
\frac{r_K}{2\max\{1,|x'_1|,|x'_2|\}}\geq \frac{r_Kr_{\Gamma_{\phi,\omega}K}}{4R_{\Gamma_{\phi,\omega}K}},
\]
holds, which then proves the conlusion of this Lemma.
\end{proof}

An immediate consequence of Lemma \ref{supfunc} and Lemma \ref{ellipsoid} reads as follows.

\begin{corollary}\label{cor}
Let $\phi\in \mathcal{C}$ be an Orlicz function, $\omega$ is a weight function on $(0,1)$ and $K\in \mathcal{K}_0^n$. If $u\in S^{n-1}$ then
$$\Gamma_{\phi,\omega}(S_uK)\subset S_u(\Gamma_{\phi,\omega}K).$$
If the inclusion is an identity for any $u\in S^{n-1}$, then $K$ is an ellipsoid centered at the origin.
\end{corollary}
With Corollary \ref{cor} in hand, we now ready prove our main theorem (i.e., Theorem \ref{maintheorem}) by using Steiner's symmetrization method.

\begin{proof}[Proof of Theorem \ref{maintheorem}]
Let $K\in \mathcal K_0^n$. By using Blaschke's selection principle, we can take a sequence of unit vectors $\{u_i\}_i \subset S^{n-1}$ such that the sequence of convex bodies defined by
\[
K_0:= K,\quad K_i=S_{u_i}K_{i-1},\qquad i\geq 1
\]
converges to $(|K|/\omega_n)^{1/n}B$ in $\mathcal K_0^n$. Thank to Lemma \ref{commute} and Lemma \ref{continue1}, we have
\begin{equation}\label{eq:limit}
\lim_{i\to\infty} \Gamma_{\phi,\omega} K_i = \left(\frac{|K|}{\omega_n}\right)^{\frac1n} \Gamma_{\phi,\omega}B \quad\text{in} \quad \mathcal K_0^n.
\end{equation}
The volume preserving property of Steiner's symmetrization and Corollary \ref{cor} imply that
\[
|\Gamma_{\phi,\omega}K_i|\leq |\Gamma_{\phi,\omega}K_{i-1}|\leq \cdots \leq |\Gamma_{\phi,\omega}K|,\qquad\forall\, i\geq 1.
\]
Letting $i\to \infty$ and using \eqref{eq:limit}, we get
\[
\frac{|K|}{\omega_n} |\Gamma_{\phi,\omega} B| = \lim_{i\to \infty} |\Gamma_{\phi,\omega}K_i| \leq |\Gamma_{\phi,\omega}K|,
\]
which then proves that the volume ratio $|\Gamma_{\phi,\omega}K|/|K|$ is minimized at $B$ in $\mathcal K_0^n$. Using again Lemma~\ref{commute}, this volume ratio is minimized at all origin--centered ellipsoids.

Suppose that the volume ratio $|\Gamma_{\phi,\omega}K|/|K|$ is minimized at $K$. By Corollary \ref{cor}, we have
\[
\Gamma_{\phi,\omega}(S_uK)\subset S_u(\Gamma_{\phi,\omega}K),\qquad\forall\, u\in S^{n-1},
\] 
which proves $|\Gamma_{\phi,\omega}(S_uK)|\leq |\Gamma_{\phi,\omega}K|$ for any $u\in S^{n-1}$. Since $|S_uK|=|K|$ and $|\Gamma_{\phi,\omega}K|/|K|$ is minimized at $K$, it holds 
\[
\frac{|\Gamma_{\phi,\omega}K|}{|K|} \leq \frac{|\Gamma_{\phi,\omega}(S_uK)|}{|S_uK|} = \frac{|\Gamma_{\phi,\omega}(S_uK)|}{|K|} \leq \frac{|\Gamma_{\phi,\omega}K|}{|K|}
\]
which forces $\Gamma_{\phi,\omega}(S_uK)= S_u(\Gamma_{\phi,\omega}K)$ for any $u\in S^{n-1}$. Thank to Corollary \ref{cor}, we conclude that $K$ is an origin--centered ellipsoid.
\end{proof}

Despite Theorem \ref{maintheorem} states only for convex bodies in $\mathcal K_0^n$, we hope that it could hold for any star bodies in $\mathcal S_0^n$. The $L_p$ Busemann--Petty centroid inequality on star bodies in $\mathcal S_0^n$ is reduced from the one on convex bodies in $\mathcal K_0^n$ by using a special class--reduction argument (see \cite{LYZ00}). We do not know the existence of this argument in proving of the Orlicz Busemann--Petty centroid inequality until now as mentioned in \cite{LYZcentroid}. Recently, Zhu \cite{Zhu12} established the Orlicz Busemann--Petty centroid inequality for all star bodies in $\mathcal S_0^n$ by extending the method of Lutwak, Yang and Zhang in \cite{LYZcentroid}. We believe that Zhu's proof could be used to extend Theorem \ref{maintheorem} to all star bodies in $\mathcal S_0^n$.

\section*{Acknowledgments}
This work was initiated when I was PhD student at Institut de Math\'ematiques de Jussieu, UPMC. I would like to thank my advisor, professor Dario Cordero--Erausquin for encouraging me to write this paper and also for his useful discussions. This work is supported by the CIMI's postdoctoral research fellowship.

\end{document}